\documentclass[12pt, twoside, fleqn, a4paper]{article}

\usepackage{latexsym}  %Extra math symbols
\usepackage{amscd}    %AMS package for doing commutative diagrams
\usepackage{theorem}  %Improved theoremlike environments
\usepackage{pifont}  %Dingbat symbols (used for end of proof symbol)
\usepackage{mathbbol} %Blackboard bold symbols
\usepackage{amsfonts}  %More fonts
\usepackage{xspace}  %Removes need for adding a \  after macros
\usepackage{amssymb} %More math symbols
\usepackage{fancyhdr} %Fancyheadings package
\usepackage{mathrsfs} %For the 3rd letters
\usepackage{amsmath} %Math symbols
\usepackage{graphics} %Picture inclusions
\usepackage{graphicx} %Picture inclusions
\usepackage{euscript}
\usepackage{psfrag}
\usepackage{empheq} %Improved equation environments 
\usepackage{mathabx} %Includes a few more stylistic arrows
\usepackage[parfill]{parskip}     %Activate to begin paragraphs with an empty line rather than an indent
\usepackage[colorlinks=true, allcolors=blue]{hyperref}
\usepackage{enumitem}

\title{Topological pressure for holomorphic correspondences using open covers }
\author{Subith Gopinathan\footnote{Indian Institute of Science Education and Research Thiruvananthapuram (IISER-TVM), Maruthamala P.O., Vithura, Kerala, India. PIN 695 551.\ \ email: \texttt{subith21@iisertvm.ac.in}, \ \texttt{subithjimail@gmail.com}} }

\DeclareFontFamily{OT1}{pzc}{}
\DeclareFontShape{OT1}{pzc}{m}{it}%
              {<-> s * [0.900] pzcmi7t}{}
\DeclareMathAlphabet{\mathpzc}{OT1}{pzc}%
                                 {m}{it}

{\theorembodyfont{\slshape} \newtheorem{theorem}{Theorem}[section]}
{\theorembodyfont{\slshape} \newtheorem{definition}[theorem]{Definition}}
{\theorembodyfont{\slshape} \newtheorem{lemma}[theorem]{Lemma}}
{\theorembodyfont{\slshape} \newtheorem{proposition}[theorem]{Proposition}}
{\theorembodyfont{\slshape} }
{\theorembodyfont{\slshape} \newtheorem{corollary}[theorem]{Corollary}} 
{\theorembodyfont{\slshape} } 
{\theorembodyfont{\slshape} }
{\theorembodyfont{\slshape} }

\numberwithin{equation}{section}

\newenvironment{proof}{\paragraph{Proof:}}{\hfill$\bullet$}

\begin{document}

\maketitle 

\begin{abstract} 
In this paper, we define the topological pressure of continuous functions with respect to the holomorphic correspondences using the open covers of the Riemann sphere. Further, we show that this method coincides with the existing definition of pressure that uses the notion of  separated and spanning family of orbits. 
\end{abstract}

\begin{tabular}{l l} 
{\bf Keywords}: & Holomorphic correspondence \\
& Open covers for holomorphic correspondence\\
& pressure for holomorphic correspondence\\
& Entropy of holomorphic correspondence \\ \\ 
& \\ 
{\bf MSC Subject} &     37D35, 37A35, 37F05 \\ 
{\bf Classifications} & \\ 
\end{tabular} 
\bigskip 

\newpage 

\section{Introduction}
 Thermodynamical formalism provides a powerful way to understand the complexity of dynamical systems. The concepts of entropy, pressure, Ruelle operator, \emph{etc}, are well explored in the case of maps. In the context of complex dynamics, holomorphic correspondences serve as a natural generalization of rational maps; hence, in recent years, the thermodynamical properties of holomorphic correspondences have emerged as a topic of interest. In \cite{ds:2008}, the entropy of holomorphic correspondences has been introduced with the help of  separated family of orbits that serve as a generalization of the concept for maps as given by Bowen in \cite{rb:1971}. Some recent papers that explore deeper aspects of the same include \cite{bs:2021, KS:2022, LLZ:2024}. Then, in \cite{SS:2025}, the authors introduced the concept of topological pressure of continuous functions with respect to holomorphic correspondences, defining it using the notions of spanning and separated families of orbits. They also develop a variational principle for holomorphic correspondences that relates the pressure and entropy in appropriate spaces. Further, they define a Ruelle operator for holomorphic correspondences and discuss the relation between its spectral properties and the pressure functional in the context. In \cite{SS:2026}, the allocation of measure-theoretic entropy for a class of measures has been chosen as the focal point of study, and this paves the way for understanding the pressure functional in a different approach. In the same paper, the authors provided an equivalent formulation of the variational principle and a measure-theoretic clarity for the convergence of the Ruelle operator. In this work, we focus on defining the topological pressure of real-valued continuous functions with respect to the holomorphic correspondences on the Riemann sphere using the open covers of the space. Further, we obtain a different formula for the entropy of holomorphic correspondence. We show that these coincide with the existing definition of the respective notions. 

\section{Preliminaries}

\begin{definition}
Let $\widehat{\mathbb{C}}$ denote the Riemann sphere. Then 
a holomorphic correspondence on $\widehat{\mathbb{C}}$ is a formal linear combination of the form
\begin{equation}
\label{correspondence}
F \;=\; \sum_{1 \le t \le N} \, F_t,
\end{equation}
where  $F_t$'s are
irreducible complex-analytic sub-varieties (repeated according to multiplicity) of pure dimension $1$ in  $\widehat{\mathbb{C}} \times \widehat{\mathbb{C}} $
that satisfy the following conditions:
\begin{enumerate}
    \item For each $F_t$, the restrictions $\pi_1|_{F_t}$ and $\pi_2|_{F_t}$ 
    are surjective;
    \item For each $x$  and $y \in \widehat{\mathbb{C}}$, the sets
    \[
    \pi_2(\pi_1^{-1}(\{x\}) \cap F_t)
    \quad\text{and}\quad
    \pi_1(\pi_2^{-1}(\{y\}) \cap F_t)
    \]
    are finite for every $1 \le t \le N$,
\end{enumerate}
where $\pi_i$ is the projection onto $\widehat{\mathbb{C}}$ for $i=1,2$.
\end{definition}

From \cite{bs:2016}, one can see that the composition of correspondences is well-defined, and from this, we can talk about the orbits arising from iterations.\\

Suppose, for $1 \leq t \leq N$ and a generic $z \in  \widehat{\mathbb{C}}$, let us denote $\delta_{t} (z) = \# \{ w \in \widehat{\mathbb{C}} : (w,z) \in \Gamma_{t} \}$ and $\lambda_{t}(z)= \# \{ w \in \widehat{\mathbb{C}} : (z, w) \in \Gamma_{t} \}$.
Now, for $k \in \mathbb{Z}_{+}$, let $\mathcal{W}_{k}$ denote the cylindrical set of all sequences of length $k$ with entries in $\{1,2,\dots,N\}$.
Then, corresponding to a given word, $\boldsymbol{\alpha} = (a_{1},\dots,a_{k}) \in \mathcal{W}_{k}$ and a point $z_{0} \in \widehat{\mathbb{C}}$, we define the collection of admissible forward paths of length $k$  by

\begin{eqnarray*} 
\mathscr{P}_{k}^{\boldsymbol{\alpha}} \left( z_0 \right) & = & \Bigg\{ \left( z_{0}, z^{(1)}_{j_{1}}, \cdots, z^{(k)}_{j_{k}};\; \alpha_{1}, \cdots, \alpha_{k} \right)\ :\ \left( z^{(r - 1)}_{j_{r - 1}}, z^{(r)}_{j_{r}} \right) \in \Gamma_{\alpha_{r}}\ \text{where}\ z^{(0)}_{j_{0}} = z_{0} , \\ 
& &\ 1 \le j_{r} \le \lambda_{\alpha_{r}} \left( x^{(r - 1)}_{j_{r - 1}} \right)\ \text{for}\ 1 \le r \le k \Bigg\}. 
\end{eqnarray*} 

With this, we define the whole set of permissible forward orbits of length $k$ generated by $F$ as
\[
\mathcal{P}_{k}^{F}(\widehat{\mathbb{C}})
=
\bigcup_{\boldsymbol{\alpha} \in \mathcal{W}_{k}}
\ \bigcup_{z_{0} \in \widehat{\mathbb{C}}}
\mathcal{P}_{k}^{\boldsymbol{\alpha}}(z_{0}).
\]
Let, 
$\mathbf{X}_{k}^{+}(z_{0};\boldsymbol{\alpha})_{\boldsymbol{j}}$ 
 denote
a general element in $\mathcal{P}_{k}^{F}(\widehat{\mathbb{C}})$, where $z_{0} \in \widehat{\mathbb{C}},\ \boldsymbol{\alpha} = \left( \alpha_{1}, \cdots, \alpha_{n} \right)$ and $\boldsymbol{j} = \left( j_{1}, \cdots, j_{n} \right)$. Assigning the product metric to $\mathcal{P}_{k}^{F}(\widehat{\mathbb{C}})$, we can talk about the following continuous functions defined from the compact space $\mathcal{P}_{k}^{F}(\widehat{\mathbb{C}})$.
\[
\Pi_i : \mathcal{P}_{k}^{F}(\widehat{\mathbb{C}}) \xrightarrow{} \widehat{\mathbb{C}} \ \text{as} \ 
\Pi_i\left( \mathbf{X}_{k}^{+}(z_{0};\boldsymbol{\alpha})_{\boldsymbol{j}}\right) = z^{(i)}_{j_{i}} \ \text{for } \ 0\leq i \leq k
\]
\[
\text{proj}_i : \mathcal{P}_{k}^{F}(\widehat{\mathbb{C}}) \xrightarrow{} \{1,2,...N\} \ \text{as} \ 
\text{proj}_i \left( \mathbf{X}_{k}^{+}(z_{0};\boldsymbol{\alpha})_{\boldsymbol{j}} \right) = \alpha_i \ \text{for } \ 1\leq i \leq k.
\]
In this work, we discuss the thermodynamical properties of holomorphic correspondences. With $d_{\widehat{\mathbb{C}}}$ denoting the spherical metric in the Riemann sphere, next we mention the concept of separated sets and spanning sets in $\mathcal{P}_{k}^{F}(\widehat{\mathbb{C}})$.

\begin{definition} \cite{SS:2025}
\label{spanning}
Let $F$ be a holomorphic correspondence defined on $\widehat{\mathbb{C}}$, as in
Equation \eqref{correspondence}, and let
$\mathcal{P}_{k}^{F}(\widehat{\mathbb{C}})$ denote the collection of all
permissible forward orbits of length $k$ generated by $F$.
Given $\epsilon > 0$, a subset
$\mathcal{V}\subseteq \mathcal{P}_{k}^{F}(\widehat{\mathbb{C}})$ is called an
\emph{$(k,\epsilon)$-spanning set} of
$\mathcal{P}_{k}^{F}(\widehat{\mathbb{C}})$ if for every
$\mathbf{X}_{k}^{+}(z_{0};\boldsymbol{\alpha})_{\boldsymbol{j}}
\in \mathcal{P}_{k}^{F}(\widehat{\mathbb{C}})$
there exists a point $w_{0} \in \widehat{\mathbb{C}}$ such that
$
\mathbf{X}_{k}^{+}(w_{0};\boldsymbol{\alpha})_{\boldsymbol{h}} \in \mathcal{V}
$
and
$
d_{\widehat{\mathbb{C}}}
\Big(
\Pi_{i}\big(\mathbf{X}_{k}^{+}(z_{0};\boldsymbol{\alpha})_{\boldsymbol{j}}),
\Pi_{i}\big(\mathbf{X}_{k}^{+}(w_{0};\boldsymbol{\alpha})_{\boldsymbol{h}} \big)
\Big)
< \epsilon
\quad \text{for all } 0 \le i \le k .
$
\end{definition}

\begin{definition}{\cite{ds:2008}}
\label{separated}
Let $F$ be a holomorphic correspondence defined on $\widehat{\mathbb{C}}$, as in
Equation \eqref{correspondence}, and let
$\mathcal{P}_{k}^{F}(\widehat{\mathbb{C}})$ be the set of all permissible
forward orbits of length $k$ generated by $F$. 
Given $\epsilon > 0$, a subset
$\mathcal{Y}\subseteq \mathcal{P}_{k}^{F}(\widehat{\mathbb{C}})$
is called \emph{$(k,\epsilon)$-separated} if for every pair of distinct points in
$\mathcal{Y}$ say
$\mathbf{X}_{k}^{+}(z_{0};\boldsymbol{\alpha})_{\boldsymbol{j}}, 
\mathbf{X}_{k}^{+}(w_{0};\boldsymbol{\beta})_{\boldsymbol{h}}$,
we have
\begin{eqnarray*}
\text{either} & &
d_{\widehat{\mathbb{C}}}
\left(
\Pi_{i}\left(\mathbf{X}_{k}^{+}(z_{0};\boldsymbol{\alpha} \right)_{\boldsymbol{j}},
\Pi_{i}\left(\mathbf{X}_{k}^{+}(w_{0};\boldsymbol{\beta})_{\boldsymbol{h}} \right)
\right)
> \epsilon
\quad \text{for some } 0 \le i \le k,
\\
\text{or} & &
\mathrm{proj}_{i+1}\left(\mathbf{X}_{k}^{+}\left(z_{0};\boldsymbol{\alpha}\right)_{\boldsymbol{j}}\right)
\neq
\mathrm{proj}_{i+1}\left(\mathbf{X}_{k}^{+}\left(w_{0};\boldsymbol{\beta}\right)_{\boldsymbol{h}}\right)
\quad \text{for some } 1 \le i \le k .
\end{eqnarray*}

\end{definition}
The terms mentioned in  Definition \ref{spanning} and \ref{separated} serve as the generalization of the analogous ideas as in the case of maps and with this information, for a continuous function $g \in \mathcal{C}\left(\widehat{\mathbb{C}},\mathbb{R}\right)$ and a given $\epsilon > 0$, we mention two more quantities as follows :

\begin{equation}
\label{R_k}
\mathscr{R}_{k}^{F}(g,\epsilon)
=
\inf
\left\{
\sum_{\mathbf{X}_{k}^{+}(z_{0};\boldsymbol{\alpha})_{\boldsymbol{j}} \in \mathcal{V}} \exp\!\left(
\sum_{i=0}^{k-1}
g\!\left(
\Pi_{i}\big(\mathbf{X}_{k}^{+}(z_{0};\boldsymbol{\alpha})_{\boldsymbol{j}}\big)
\right)
\right) : \mathcal{V}\subseteq \mathcal{P}_{k}^{F}(\widehat{\mathbb{C}})
\text{ is } (k,\epsilon)\text{-spanning}
\right\}.
\end{equation}

\begin{equation}
\label{S_k}
\mathscr{S}_{k}^{F}(g,\epsilon)
=
\sup
\left\{
\sum_{\mathbf{X}_{k}^{+}(z_{0};\boldsymbol{\alpha})_{\boldsymbol{j}} \in \mathcal{Y}}
\exp\!\left(
\sum_{i=0}^{k-1}
g\!\left(
\Pi_{i}\big(\mathbf{X}_{k}^{+}(z_{0};\boldsymbol{\alpha})_{\boldsymbol{j}}\big)
\right)
\right)
:
\mathcal{Y} \subseteq \mathcal{P}_{k}^{F}(\widehat{\mathbb{C}})
\text{ is } (k,\epsilon)\text{-separated}
\right\}.
\end{equation}

This helps in formulating the concept of topological pressure in the context of holomorphic correspondence as follows:

\begin{theorem} \cite{SS:2025}
\label{Pressure charac theorem}
Let $F$ be a holomorphic correspondence defined on $\widehat{\mathbb{C}}$, as represented in
Equation \eqref{correspondence}.
Then, for any function $g \in \mathcal{C}(\widehat{\mathbb{C}},\mathbb{R})$, the topological
pressure of $g$ with respect to $F$ is given by
\[
\mathrm{Pr}(g,F)
=
\lim_{\epsilon \to 0}
\left(
\limsup_{k \to \infty}
\frac{1}{k}
\log \mathscr{S}_{k}^{F}(g,\epsilon)
\right)
=
\lim_{\epsilon \to 0}
\left(
\limsup_{k \to \infty}
\frac{1}{k}
\log \mathscr{R}_{k}^{F}(g,\epsilon)
\right).
\]
\end{theorem}

\begin{corollary} \cite{SS:2025}
    Let $F$ be a holomorphic correspondence defined on $\widehat{\mathbb{C}}$, as represented in Equation \eqref{correspondence}, then the topological entropy of the holomorphic correspondence $F$ is given by $h_{\mathrm{top}}(F) = \mathrm{Pr}(0,F)$. 
\end{corollary}
 In this paper, we develop an alternative approach to define the pressure functional.

 \section {Open covers for holomorphic correspondence}

 Suppose, ${P} = \left\{ P_{1}, P_{2}, \cdots, P_{s} \right\}$ and $Q = \left\{ Q_{1}, Q_{2}, \cdots, Q_{s'} \right\}$ form a collection of subsets of
$\mathscr{P}^{\Gamma} \left(\widehat{\mathbb{C}} \right)$. 

We define
\[
P \bigvee Q
=
\left\{
P_{i} \cap Q_{i'}
:\ 
P_{i} \in \mathcal{P}
\ \text{and}\ 
Q_{i'} \in \mathcal{Q}
\ \text{for}\ 
1 \le i \le s,\ 
1 \le i' \le s'
\right\}.
\]
Now, let  $\mathcal{U}$ be an open cover for $\widehat{\mathbb{C}}$ and  $g \in \mathcal{C}(\widehat{\mathbb{C}},\mathbb{R})$. Then we define,

\begin{equation}
\label{M_k}
    \mathcal{M}_{k}^F(g,\mathcal{U}) = \inf_{\mathcal C} \left\{ \sum_{\gamma \in \mathcal C }\left( \inf_{\mathbf{X}_{k}^{+}(z_{0};\boldsymbol{\alpha})_{\boldsymbol{j}} \in \mathcal{\gamma}} \left\{\exp\!\left(
\sum_{i=0}^{k-1}
g\!\left(
\Pi_{i}\big(\mathbf{X}_{k}^{+}(z_{0};\boldsymbol{\alpha})_{\boldsymbol{j}}\big)
\right)\right) \right\} \right)   \right\}
\end{equation} and

\begin{equation}
    \label{N_k}
    \mathcal{N}_{k}^F(g,\mathcal{U}) = \inf_{\mathcal C} \left\{ \sum_{\gamma \in \mathcal C }\left( \sup_{\mathbf{X}_{k}^{+}(z_{0};\boldsymbol{\alpha})_{\boldsymbol{j}} \in \mathcal{\gamma}} \left\{\exp\!\left(
\sum_{i=0}^{k-1}
g\!\left(
\Pi_{i}\big(\mathbf{X}_{k}^{+}(z_{0};\boldsymbol{\alpha})_{\boldsymbol{j}}\big)
\right)\right) \right\} \right)   \right\}
\end{equation}
where  $\mathcal{C}$ denote a finite sub-cover of $\displaystyle\bigvee_{i=0}^{k-1} \left\{\Pi_{i}^{-1}(\mathcal{U} )\cap \text{proj}_{i+1}^{-1} \{j\} : 1 \leq j \leq N  \right\} \bigvee \Pi_{k}^{-1}(\mathcal{U}) $, in both Equations \eqref{M_k} and \eqref{N_k}.\\
  
One may recall that a positive number $\delta > 0$ is called a Lebesgue number for the open cover $\mathcal G$ of a metric space $(X,d)$ if for every point $x \in X$, there exists a set $G \in \mathcal{G}$ such that
$
B(x,\delta) \subset G,
$
where $B(x,\delta) = \{y \in X : d(x,y) < \delta\}$ is the open ball centered at $x$ with radius $\delta$.

Here, we state a proposition that relates the quantities defined in Equations \eqref{R_k},\eqref{S_k}, \eqref{M_k}, and \eqref {N_k}.
 
 \begin{proposition}
\label{prop}
 \begin{enumerate}[label=(\alph*)]
     \item If $\mathcal{U}$ is an open cover of $\widehat{\mathbb{C}}$ with  Lebesgue number $\delta$, then for $g \in \mathcal{C}(\widehat{\mathbb{C}},\mathbb{R})$, we have $\mathcal{M}_{k}^{F}(g,\mathcal{U}) \leq \mathscr{R}_{k}^{F}(g,\frac{\delta}{2}) \leq \mathscr{S}_{k}^{F}(g, \frac{\delta}{2}).$
\item If $\epsilon > 0$ and $\mathcal{Z}$ is an open cover of $\widehat{\mathbb{C}}$ with $diam(\mathcal{Z} )\leq \epsilon$, then for $g \in \mathcal{C}(\widehat{\mathbb{C}},\mathbb{R})$, we have $\mathscr{R}_{k}^{F}(g,\epsilon) \leq\mathscr{S}_{k}^{F}(g, \epsilon)
\leq \mathcal{N}_{k}^{F}(g,\mathcal{Z}).$
\end{enumerate}
\end{proposition}
\begin{proof}\begin{enumerate}[label=(\alph*)]
        \item   Given, $\mathcal{U}$ is an open cover of $\widehat{\mathbb{C}}$. Then for each $k \in \mathbb{Z}_+$, there exists an open cover $\mathcal{U}_{k}^F$ for $\mathcal{P}_{k}^{F}(\widehat{\mathbb{C}})$ given by  $$\mathcal{U}_{k}^F = \displaystyle\bigvee_{i=0}^{k-1} \left\{ \Pi_{i}^{-1}(\mathcal{U} )\cap \text{proj}_{i+1}^{-1} \{j\} : 1 \leq j \leq N  \right\} \bigvee  \Pi_{k}^{-1}(\mathcal{U}) .$$ 
 \\
Since $\delta$ is a Lebesgue number associated with the the open cover $\mathcal U$, for every
$z\in\widehat{\mathbb C}$ there exists $U\in\mathcal U$ such that the open ball of radius $\delta$ around $z$ is contained in $U$, i.e.,\ 
$
B(z,\delta)\subset U.
$ . We claim that $\delta$ is also a Lebesgue number associated with the open cover $\mathcal{U}_{k}^F$.  \\
Suppose  
$\mathbf{X}_{k}^{+}(z_{0};\boldsymbol{\alpha})_{\boldsymbol{j}}
=
\left( z_{0}, z^{(1)}_{j_{1}}, \cdots, z^{(k)}_{j_{k}};\; \alpha_{1}, \cdots, \alpha_{k} \right) 
\in \mathcal P^{F}_{k}(\widehat{\mathbb C}),$ then for $0\le i\le k$, we have $$\Pi_i\big(B(\mathbf{X}_{k}^{+}(z_{0};\boldsymbol{\alpha})_{\boldsymbol{j}},\delta)\big)\subset B(z^{(i)}_{j_{i}},\delta).$$
Hence there exists $U_i\in\mathcal U$ such that
$
\Pi_i\big(B(\mathbf{X}_{k}^{+}(z_{0};\boldsymbol{\alpha})_{\boldsymbol{j}}),\delta)\big)\subset U_i,
$ which implies $$
B(\mathbf{X}_{k}^{+}(z_{0};\boldsymbol{\alpha})_{\boldsymbol{j}},\delta)\subset \Pi_i^{-1}(U_i), \ \text{for} \ 0\leq i \leq k.$$
Hence, for $\delta < 1, $ there exists 
$
 U^F \in \displaystyle\bigvee_{i=0}^{k-1} \left\{ \Pi_{i}^{-1}(\mathcal{U} )\cap \text{proj}_{i+1}^{-1} \{j\} : 1 \leq j \leq N  \right\} \bigvee  \Pi_{k}^{-1}(\mathcal{U}) 
$ such that  
$$
B(\mathbf{X}_{k}^{+}(z_{0};\boldsymbol{\alpha})_{\boldsymbol{j}},\delta)
\subset
U^F. 
$$ Thus every ball of radius $\delta$ in
$\mathcal P^{F}_{k}(\widehat{\mathbb C})$
is contained in some element of $\mathcal U_K^{F}$. Thus, $\delta$ is also a Lebesgue number associated with the cover $\mathcal{U}_{k}^F$ \\
Suppose, $\mathcal{V}$ is a $(k,\frac{\delta}{2})$ spanning set of $\mathcal{P}_{k}^{F}(\widehat{\mathbb{C}})$, then we can write
        $$\mathcal{P}_{k}^{F}(\widehat{\mathbb{C}}) = \bigcup_{\mathbf{X}_{k}^{+}(z_{0};\boldsymbol{\alpha})_{\boldsymbol{j}} \in \mathcal{V}} \overline{B\left(\mathbf{X}_{k}^{+}(z_{0};\boldsymbol{\alpha})_{\boldsymbol{j}},\frac{\delta}{2}\right)} .
     $$
            Since $\delta$ is a Lebesgue number assosciated with the cover $\mathcal{U}_{k}^F$, each $\overline{B\left(\mathbf{X}_{k}^{+}(z_{0};\boldsymbol{\alpha})_{\boldsymbol{j}}),\frac{\delta}{2}\right)}$ is a subset of a member of $\mathcal{U}_{k}^F$ . This gives us, $$\mathcal{M}_{k}^{F}(g,\mathcal{U}) \leq \mathscr{R}_{k}^{F}(g,\frac{\delta}{2}).$$
Finally, one can see from the proof of Lemma $3.4$ in \cite{SS:2025},  that a $(k, \frac{\delta}{2})$ separated family of maximum cardinality is also a $(k, \frac{\delta}{2})$ spanning family of $\mathcal{P}_{k}^{F}(\widehat{\mathbb{C}})$ , and therefore we can conclude that $$\mathcal{M}_{k}^{F}(f,\mathcal{U}) \leq \mathscr{R}_{k}^{F}(f,\frac{\delta}{2}) \leq \mathscr{S}_{k}^{F}(f,\frac{\delta}{2}).$$
\item 
 Let $\mathcal{Y} $ be a $(k,\epsilon)$ separated subset of $\mathcal{P}_{k}^{F}(\widehat{\mathbb{C}})$.    Since, $diam(\mathcal{Z}) < \epsilon $, from the construction of corresponding $\mathcal{Z}_{k}^F$, we can see that $ diam (\mathcal{Z}_{k}^F) < \epsilon $. Then, no member of  $\mathcal{Z}_{k}^F$ contains two elements of $\mathcal{Y}$. Therefore, $$\mathscr{S}_{k}^{F}(g, \epsilon)
\leq \mathcal{N}_{k}^{F}(g,\mathcal{Z}).$$  Hence the proof.

    \end{enumerate}
\end{proof}
\begin{proposition} 
 \label{limit}
Let $g \in C(\widehat{\mathbb{C}},\mathbb{R})$ and let $\mathcal{U}$ be an open cover of 
$\widehat{\mathbb{C}}$. Then 
$$
\lim_{k\to\infty} \frac{1}{k}\log \mathcal{N}_{k}^{F}(g,\mathcal{U})
$$
exists and equals to 
$$
\inf_{k\ge 1} \frac{1}{k}\log \mathcal{N}_{k}^{F}(g,\mathcal{U}).
$$
\end{proposition}

 Here, we state a lemma that helps in proving the above proposition.
\begin{lemma} \cite{pw:1982}
    \label{sequence}
Let $(a_n)_{n \ge 1}$ be a sequence of real numbers such that $a_{n+p} \le a_n + a_p \quad \text{for all } \\ n,p \ge 1$.
Then the limit $ \displaystyle\lim_{n \to \infty} \frac{a_n}{n} $
exists and
$
\displaystyle\lim_{n \to \infty} \frac{a_n}{n}
=
\inf_{n \ge 1} \frac{a_n}{n}.
$
\end{lemma}
Now, we give the proof for Proposition \ref{limit}.
\begin{proof}
To prove this proposition, it is enough to show that 
$$
\mathcal{N}^{F}_{n+k}(g,\mathcal{U})
\le
\mathcal{N}^{F}_{n}(g,\mathcal{U}) \cdot \mathcal{N}_{k}^{F}(g,\mathcal{U}),
$$ and the proof follows from  Lemma
\ref{sequence}. Let $\mathcal{O}$ and let $\mathcal{Y}$ be  finite subcovers of the spaces
\begin{align*}
&\left\{
\bigvee_{i=0}^{n-1}
\left\{
\Pi_i^{-1}(\mathcal{U}) \cap \operatorname{proj}_i^{-1}\{j\}
:\; 1 \le j \le N
\right\}
\;\bigvee\;
\Pi_{n}^{-1}(\mathcal{U})
\right\}
\\[6pt]
\text{and}\quad
&\left\{
\bigvee_{i=0}^{k-1}
\left\{
\Pi_i^{-1}(\mathcal{U}) \cap \operatorname{proj}_i^{-1}\{j\}
:\; 1 \le j \le N
\right\}
\;\bigvee\;
\Pi_{k}^{-1}(\mathcal{U})
\right\}.
\end{align*}
\text{respectively.} \\
Now, we define  for $O_i \in \mathcal{O} \ \text{and} \ Y_h \in \mathcal{Y}$ 
\begin{align*}
O_i \star Y_h = \{(z_{0},z_{j_1}^{(1)},...z_{j_n}^{(n)},z_{j_{n+1}}^{(n+1)},...,z_{j_{k+n}}^{(k+n)}: \alpha_0,\alpha_1,...\alpha_n,\alpha_{n+1},...\alpha_{k+n}):\\ (z_{0},z_{j_1}^{(1)},...z_{j_n}^{(n)}:\alpha_1,...\alpha_n) \in O_i \ \ \text{and} \ (z_{j_n}^{(n)},z_{j_{n+1}}^{(n+1)},...,z_{j_{k+n}}^{(k+n)};\alpha_{n+1},...\alpha_{k+n}) \in Y_h\}, 
\end{align*} 
It is important to see that the orbit family $O_i \star Y_h$ is valid only if the above condition holds between the appropriate orbits of $O_i$ and $Y_h$. With this we denote $$\mathcal{O} \bigstar \mathcal{Y} = \{O_i \star Y_h : O_i \in \mathcal{O} \ \text{and} \  Y_h \in \mathcal{Y}\}.$$ Then, we can observe that
$\mathcal{O} \bigstar \mathcal{Y}$ is a finite  sub-cover of $$\displaystyle \bigvee_{i=0}^{k+n-1} \left\{
\Pi_i^{-1}(\mathcal{U}) \cap \operatorname{proj}_i^{-1}\{j\}
:\; 1 \le j \le N
\right\} \bigvee \Pi_{k+n}^{-1}(\mathcal U). $$
This, in turn, shows
\begin{align}
&\sum_{ O_i \star Y_{j}\in \mathcal{O}\bigstar \mathcal{Y} }
\left( \sup_{\substack{
\mathbf{X}_{n+k}^{+}(z_{0};\boldsymbol{\alpha})_{\boldsymbol{j}}\\
\in O_i \star Y_{j}}}
\exp\left(
\sum_{r=0}^{n+k-1}
g\!\left(
\Pi_{r}
\big(
\mathbf{X}_{n+k}^{+}(z_{0};\boldsymbol{\alpha})_{\boldsymbol{j}}
\big)
\right)
\right) \right)
\nonumber\\
&\le
\left(
\sum_{ O_i\in \mathcal{O} }
\left( \sup_{\mathbf{X}_{n}^{+}(z_{0};\boldsymbol{\alpha}')_{\boldsymbol{j}'}
\in O_i}
\exp\left(
\sum_{r=0}^{n-1}
g\!\left(
\Pi_{r}
\big(
\mathbf{X}_{n}^{+}(z_{0};\boldsymbol{\alpha}')_{\boldsymbol{j}'}
\big)
\right)
\right)
\right)\right)
\nonumber\\
&\qquad \times
\left(
\sum_{ Y_{j}\in \mathcal{Y}}
\left(\sup_{\mathbf{X}_{k}^{+}(z_{n};\boldsymbol{\alpha}'')_{\boldsymbol{j}''}
\in Y_{\boldsymbol{j}}}
\exp\left(
\sum_{r=0}^{k-1}
g\!\left(
\Pi_{r}
\big(
\mathbf{X}_{k}^{+}(z_{n};\boldsymbol{\alpha}'')_{\boldsymbol{j}''}
\big)
\right)
\right)
\right) \right).
\end{align}

where 
\[
\mathbf{X}_{n+k}^{+}(z_{0};\boldsymbol{\alpha})_{\boldsymbol{j}}
=  \left( z_{0},z_{j_1}^{(1)},...z_{j_n}^{(n)},z_{j_{n+1}}^{(n+1)},...,z_{j_{k+n}}^{(k+n)} ; \alpha_0,\alpha_1,...\alpha_n,\alpha_{n+1},...\alpha_{k+n} \right)
\]
\[
\mathbf{X}_{n}^{+}(z_{0};\boldsymbol{\alpha'})_{\boldsymbol{j}'}
=
\left(z_{0},z_{j_1}^{(1)},...,z_{j_n}^{(n)}:
\alpha_{1},\dots,\alpha_{n} \right), \text{and}
\] 
\[
\mathbf{X}_{k}^{+}(z_{n};\boldsymbol{\alpha}'')_{\boldsymbol{j}''}
=
\left(z_{j_n}^{(n)},z_{j_{n+1}}^{(n+1)},...,z_{j_{k+n}}^{(k+n)};
\alpha_{n+1},\dots,\alpha_{k+n}\right).
\]

This in turn gives us $\mathcal{N}_{n+k}^F(g,\mathcal{U})
\le
\mathcal{N}_{n}^F(g,\mathcal{U}) \cdot \mathcal{N}_{k}^{F}(g,\mathcal{U})$
Hence, the proof.
\end{proof}

 Now, we move to the main theorem of the section.

 \begin{theorem}

 \label{PRESSURE THEOREM}
Let $F$ be a holomorphic correspondence on $\widehat{\mathbb{C}}$ represented as in Equation \eqref{correspondence} and let
$g \in \mathcal{C}(\widehat{\mathbb{C}},\mathbb{R})$. Then the topological pressure of $g$ with respect
to $F$ satisfies
\[
\mathrm{Pr}(g,F) =
\lim_{\delta \rightarrow{0} } \sup_{\mathcal U}
\lim_{k\to\infty}
\frac{1}{k}\log \mathcal{N}_k^F(g,\mathcal U)=\lim_{\delta \rightarrow{0} } \sup_{\mathcal U}
\lim_{k\to\infty}
\frac{1}{k}\log \mathcal{M}_k^F(g,\mathcal U),
\]
where the supremum is taken over all open covers $\mathcal U$ of $\widehat{\mathbb{C}}$ with $\mathrm{diam}(\mathcal{U}) < \delta$.
\end{theorem}

\begin{proof}
Let $\delta>0$ and let $\mathcal{U}$ be an open cover of $\widehat{\mathbb{C}}$ such that
$\mathrm{diam}(\mathcal{U}) < \delta$. By Proposition \ref{prop}(b), we have for every
$k \geq 1$,
\[
\mathscr{S}_k^F(g,\delta)
\le
\mathcal{N}^F(g,\mathcal U).
\]
This implies,
\[
\lim_{k\to\infty}
\frac{1}{k}\log \mathscr{S}_k^F(g,\delta)
\le
\lim_{k\to\infty}
\frac{1}{k}\log \mathcal{N}_k^F(g,\mathcal U).
\]
Taking the supremum over all open covers $\mathcal{U}$ with
$\mathrm{diam}(\mathcal U) < \delta$ and then letting $\delta\to0$, we get
\[
\mathrm{Pr}(g,F)
\le
\lim_{\delta \to 0}\sup_{\mathcal U}
\lim_{k\to\infty}
\frac{1}{k}\log \mathcal{N}_k^F(g,\mathcal U).
\]

Conversely, let $\mathcal U$ be an arbitrary open cover of $\widehat{\mathbb{C}}$ and let
$\delta>0$ be a Lebesgue number associated with $\mathcal{U}$. By
Proposition \ref{prop}(a), for every $k\ge1$,
\[
\mathcal M_k^F(g,\mathcal U)
\le
\mathscr S_k^F\!\left(g,\frac{\delta}{2}\right).
\]
Now, if 
$\tau := \sup \{ |g(z)-g(w)| : d(z,w) \le \operatorname{diam}(\mathcal U) < 1\}  $\ , then 

for two $k$- long orbits say, $ \mathbf{X}_{k}^{+}(z_{0};\boldsymbol{\alpha})_{\boldsymbol{j}} \ \text{and} \ \mathbf{X}_{k}^{+}(w_{0};\boldsymbol{\alpha})_{\boldsymbol{h}} $ lying in the same set in the finite subcover of corresponding $\mathcal{U}^F_K$, we have
\[
\left|
\sum_{r=0}^{k-1}
g\!\left(\Pi_r(\mathbf{X}_{k}^{+}(z_{0};\boldsymbol{\alpha})_{\boldsymbol{j}})\right)
-
\sum_{r=0}^{k-1}
g\!\left(\Pi_r(\mathbf{X}_{k}^{+}(w_{0};\boldsymbol{\alpha})_{\boldsymbol{h}})\right)
\right|
\le k\tau.
\]

By incorporating the appropriate factors and limits, we get, 
\begin{equation}
    \label{inequality}
\mathcal N_k^F(g,\mathcal U)
\le\exp
({k\tau})
\mathcal M_k^F(g,\mathcal U).
\end{equation}
Consequently,
$$
\sup_{\mathcal U}
\lim_{k\to\infty}
\frac{1}{k}\log \mathcal{N}_k^F(g,\mathcal U) 
\le \tau +\sup_{\mathcal U}
\lim_{k\to\infty}
\frac{1}{k}\log \mathcal{M}_k^F(g,\mathcal U) \le
\tau
+
\lim_{k\to\infty}\frac{1}{k}
\log
\mathscr S_k^F\!\left(g,\frac{\delta}{2}\right),
$$
where the supremum is taken over all open covers $\mathcal U$ of $\widehat{\mathbb{C}}$ with $\mathrm{diam}(\mathcal{U}) < \delta$.\\
As $\delta \rightarrow{0} , \ \text{we get} \  \tau \rightarrow{0}$ and hence,
\[
\lim_{\delta \rightarrow{0} } \sup_{\mathcal U}
\lim_{k\to\infty}
\frac{1}{k}\log \mathcal{N}_k^F(g,\mathcal U) \le \mathrm{Pr}(g,F). 
\]

Hence, we get 
\begin{equation}
\label{pressure form}
\mathrm{Pr}(g,F) =
\lim_{\delta \rightarrow{0} } \sup_{\mathcal U}
\lim_{k\to\infty}
\frac{1}{k}\log \mathcal{N}_k^F(g,\mathcal U) 
\end{equation}

Now, from Equations \eqref{M_k}, \eqref{N_k} and \eqref{inequality} and ,we have \[
\exp
({-k\tau})\mathcal N_k^F(g,\mathcal U)
\le
\mathcal M_k^F(g,\mathcal U)\le \mathcal N_k^F(g,\mathcal U)
\]
This gives,
$$ 
-\tau + \lim_{k\to\infty}
\frac{1}{k}\log \mathcal{N}_k^F(g,\mathcal U) \le \lim_{k\to\infty}
\frac{1}{k}\log \mathcal{M}_k^F(g,\mathcal U) \le \lim_{k\to\infty}
\frac{1}{k}\log \mathcal{N}_k^F(g,\mathcal U) 
$$

Then, from Sandwich Theorem and Equation \eqref{pressure form}, we get 
$$
\mathrm{Pr}(g,F) =
\lim_{\delta \rightarrow{0} } \sup_{\mathcal U}
\lim_{k\to\infty}
\frac{1}{k}\log \mathcal{N}_k^F(g,\mathcal U)=\lim_{\delta \rightarrow{0} } \sup_{\mathcal U}
\lim_{k\to\infty}
\frac{1}{k}\log \mathcal{M}_k^F(g,\mathcal U),
$$
where the supremum is taken over all open covers $\mathcal U$ of $\widehat{\mathbb{C}}$ with $\mathrm{diam}(\mathcal{U}) < \delta$.
\end{proof}

 \begin{corollary}
     Let $F$ be a holomorphic correspondence on $\widehat{\mathbb{C}}$ represented as in Equation \eqref{correspondence}, then the topological entropy of $F$ denoted by $h_\mathrm{top}(F)$ is given by
$$
h_\mathrm{top}(F) = \lim_{\delta \rightarrow{0} } \sup_{\mathcal U}
\lim_{k\to\infty}
\frac{1}{k}\log \mathcal{N}_k^F(0,\mathcal U) =  \lim_{\delta \rightarrow{0} } \sup_{\mathcal U}
\lim_{k\to\infty}
\frac{1}{k}\log \max \{\text{ \# \ of a finite subcover of} \ \  \mathcal{U}_{k}^F \}
$$
where the supremum is taken over all open covers $\mathcal U$ of $\widehat{\mathbb{C}}$ with $\mathrm{diam}(\mathcal{U}) \leq\delta$.
     \\
     \end{corollary}

\textbf{Acknowledgment:} I would like to express my sincere gratitude to my thesis supervisor, Dr. Shrihari Sridharan, for the valuable discussions on the subject.

\end{document}